\documentclass[10pt,letter]{amsart}
\usepackage[dvips]{graphicx}
\usepackage{amsthm, amssymb, color}
\usepackage{tikz}
\usepackage{amsfonts, amscd}
\usepackage{epsfig,multicol}
\usepackage{comment}
\usepackage{url}

\usepackage[arrow,matrix]{xy}
\theoremstyle{plain} \numberwithin{equation}{section}
\newtheorem{theo}{Theorem}[section]
\newtheorem{coro}[theo]{Corollary}
\newtheorem{prop}[theo]{Proposition}
\newtheorem{lemm}[theo]{Lemma}

\theoremstyle{definition}
\newtheorem*{defi}{Definition}
\newtheorem*{exam}{Example}
\newtheorem*{rema}{Remark}

\def\Z{\mathbb Z}
\def\C{\mathbb C}
\def\R{\mathbb R}
\def\Q{\mathbb Q}

\def\q{\pi}
\def\v{v}
\def\tv{\tilde\v}
\def\bv{\bar\v}
\def\hN{\hat N}
\def\m{\mu}
\def\d{d}
\def\bQ{\tilde{Q}}
\def\qn{q_n^-}
\def\Qn{Q_n}
\def\Xn{X_n}

\DeclareMathOperator{\rank}{rank}

\DeclareMathOperator{\Hom}{Hom}

\DeclareMathOperator{\Tor}{Tor}
\DeclareMathOperator{\Int}{Int}

\DeclareMathOperator{\coker}{coker}

\newcommand{\brown}{}
\newcommand{\red}{}
\newcommand{\blue}{}

\begin{document}
\title{\blue{Torsion} in the cohomology of torus orbifolds}
\author[H. Kuwata]{Hideya Kuwata}
\address{Department of Mathematics, Osaka City University, Sumiyoshi-ku, Osaka 558-8585, Japan.}
\email{hideya0813@gmail.com}

\author[M. Masuda]{Mikiya Masuda}
\address{Department of Mathematics, Osaka City University, Sumiyoshi-ku, Osaka 558-8585, Japan.}
\email{masuda@sci.osaka-cu.ac.jp}

\author[H. Zeng]{Haozhi Zeng}
\address{Department of Mathematics, Osaka City University, Sumiyoshi-ku, Osaka 558-8585, Japan.}
\address{Current address: School of Mathematical Sciences, Fudan University, Shanghai, 200433, P.R. China.}
\email{zenghaozhi@icloud.com}

\date{\today}
\thanks{The second author was partially supported by JSPS Grant-in-Aid for Scientific Research 25400095}
\subjclass[2000]{Primary 55N10, 57S15; Secondary 14M25}
\keywords{toric orbifold, cohomology, torsion.}

\begin{abstract}
We study \red{torsion} in the integral cohomology of a certain family of $2n$-dimensional orbifolds $X$ with actions of the $n$-dimensional compact torus. Compact simplicial toric varieties are in our family. For a prime number $p$, we find a necessary condition for the integral cohomology of $X$ to have no $p$-torsion. Then we prove that the necessary condition is sufficient in some cases. We also give an example of $X$ which shows that the necessary condition is not sufficient in general.
\end{abstract}
\maketitle 

\section*{Introduction}

A toric variety is a normal complex algebraic variety of complex dimension $n$ with an algebraic action of $(\C^*)^n$ having a dense orbit. A toric variety is not necessarily compact and may have singularity. The famous theorem of Danilov-Jurkiewicz gives an explicit description of the integral cohomology ring of a compact smooth toric variety in terms of the associated fan. It in particular says that the integral cohomology groups are torsion-free and concentrated in even degrees. 

The analogous result holds for a compact simplicial toric variety $X$ (simplicial means that $X$ is an orbifold) but with rational coefficients. 
S. Fischli and A. Jordan studied the integral cohomology groups $H^*(X)$
in their dissertations \cite{fisc92}, \cite{jord98} using spectral
sequences. Their results give an explicit computation of $H^k(X)$ and
$H^{2n-k}(X)$ for $k\le 3$ under some conditions. Based on their
results, M. Franz \brown{developed Maple package torhom} \cite{fran04} to compute those cohomology groups. One can see that $H^*(X)$ has torsion in general while it has no torsion when $X$ is a weighted projective space (\cite{kawa73}). Therefore we are naturally led to ask when $H^*(X)$ has torsion or no torsion. 

The orbit space $Q$ of a compact simplicial toric variety $X$ by the restricted action of the $n$-dimensional compact torus $T$ is a nice manifold with corners (sometimes called a manifold with faces). All faces of $Q$ (even $Q$ itself) are contractible and $Q$ is often homeomorphic to a simple polytope as manifolds with corners. MacPherson showed that $X$ is homeomorphic to the quotient space $(Q\times T)/\!\sim$ under some equivalence relation $\sim$ defined using the primitive vectors in the one-dimensional cones in the fan of $X$ (see \cite{fran10}). The one-dimensional cones correspond to the facets of $Q$ so that one can think of the primitive vectors as a map 
$$\v\colon \{Q_1,Q_2,\dots,Q_m\}\to \Z^n\qquad\text{($Q_i$'s are facets of $Q$). }$$ 
The map $\v$ satisfies some linear independence condition and a map satisfying the condition is called a \emph{characteristic function} on $Q$ (see Definition in Section~\ref{sect:1}). Note that there are many characteristic functions which do not arise from compact simplicial toric varieties. 

Bahri-Sarkar-Song \cite{ba-sa-so15} consider the quotient space $X(Q,\v)=(Q\times T)/\!\!\sim$. Although they restrict their concern to $Q$ being a simple polytope, the characteristic function $\v$ used to define the equivalence relation $\sim$ is arbitrary; so the quotient space \blue{does} not necessarily arise from a compact simplicial toric variety. They give a sufficient condition for $H^*(X(Q,\v))$ to be torsion-free in terms of $Q$ and $\v$. 
\red{They also give a Danilov-Jurkiewicz type description for the ring structure of $H^*(X(Q,\v))$ when it is torsion-free.}

In this paper, we also consider the quotient space $X=X(Q,\v)=(Q\times T)/\!\!\sim$ where $\v$ is arbitrary as above but our $Q$ is a \red{compact} connected nice manifold with corners and not necessarily a simple polytope. When $Q$ has a vertex (equivalently $X$ has a $T$-fixed point), our $X$ is a torus orbifold in the sense of \cite{ha-ma03}. 
We give an explicit description of $H^{k}(X)$ and $H^{2n-k}(X)$ for $k\le 2$ under some condition on $Q$. Motivated by the explicit description of $H^{2n-1}(X)$, we introduce a positive integer $\m(Q_I)$ depending on the characteristic function $\v$ for each $Q_I=\bigcap_{i\in I}Q_i$, where $I$ is a subset of $\{1,\dots,m\}$ and we understand $Q_I=Q$ when $I=\emptyset$ and $\m(Q_I)=1$ when $Q_I=\emptyset$. The $\m(Q_I)$'s are all one when $X$ has no singularity. Here is a summary of our results, which follows from Propositions~\ref{prop:3-1},~\ref{prop:5-1},~\ref{prop:5-2} and~\ref{prop:5-4}. 

\medskip
\noindent
{\bf Theorem.} 
{\it \red{Let $Q$ be a connected nice manifold with corners of dimension $n\ge 1$.} 
Let $p$ be a prime number and suppose that every face of $Q$ (even $Q$ itself) is acyclic with \blue{$\Z/p$-coefficients.} If $H^*(X(Q,\v))$ has no $p$-torsion, then $\m(Q_I)$ is coprime to $p$ for every $Q_I$. The converse holds when the face poset of $Q$ is isomorphic to the face poset of one of the following:
\begin{enumerate}
\item the suspension \red{$\lozenge^n$} of the $(n-1)$-simplex $\Delta^{n-1}$, i.e. \red{$\lozenge^n$} is obtained from $\Delta^{n-1}\times [-1,1]$ by collapsing $\Delta^{n-1}\times \{1\}$ and $\Delta^{n-1}\times \{-1\}$ to a point respectively,
\item $\Delta^n$,
\item $\Delta^{n-1}\times [-1,1]$.
\end{enumerate}
}

\begin{rema}
\red{(1) When $n\ge 3$, there are many nice manifolds with corners $Q$ which have the same face posets as $\lozenge^n$, $\Delta^n$ or $\Delta^{n-1}\times[-1,1]$ but not homeomorphic to them. For instance, one can produce such $Q$ by taking connected sum of them and integral homology $n$-spheres with non-trivial fundamental groups.}

(2) The $n$-simplex $\Delta^n$ and the prism $\Delta^{n-1}\times[-1,1]$ can be obtained from the suspension \red{$\lozenge^n$} by performing a vertex cut once and twice respectively. So, the reader might think that the converse mentioned in the theorem above would hold for $Q$ obtained from \red{$\lozenge^n$} by performing a vertex cut repeatedly. However, we will see in Section~\ref{sect:6} that this is not true for $Q$ obtained from \red{$\lozenge^3$} by performing a vertex cut four times. 
\end{rema}

The paper is organized as follows. In Section~\ref{sect:1} we set up
notations. 
In Section~\ref{sect:2} we compute $H^{2n-k}(X)$ $(k\le 2)$ for the
quotient space $X=(Q\times T)/\!\!\sim$ using the idea in Yeroshkin's
paper \cite{ye15}.
Namely, we delete a small neighborhood of the singular set in $X$ to obtain a smooth manifold and investigate the relation of the cohomology groups between $X$ and the smooth manifold. 
In Section~\ref{sect:2-1} we show that the quotient map $X\to Q$ induces
an isomorphism on their fundamental groups when $Q$ has a vertex. In
Section~\ref{sect:2-2} we apply the results in Sections~\ref{sect:2}
and~\ref{sect:2-1} to the case when $n=2$ and $3$.
In Section~\ref{sect:3} we introduce $\m(Q_I)$ and find a necessary condition for $H^{*}(X)$ to have no $p$-torsion. In Section~\ref{sect:4} we recall Theorem on Elementary Divisors and deduce two facts used in Section~\ref{sect:5}. In Section~\ref{sect:5} we prove that the necessary condition obtained in Section~\ref{sect:3} is sufficient for $Q$ mentioned in the theorem above. Section~\ref{sect:6} gives an example mentioned in the remark above. 
In the appendix we will observe that a result of Fischli or Jordan on $H^{2n-1}(X)$ and the torsion part of $H^{2n-2}(X)$ agrees with our Proposition~\ref{prop:2-1} when $X$ is a compact simplicial toric variety.

\section{Setting and notation} \label{sect:1}

In this section, we set up some notations and give some remarks. Let $Q$
be a connected manifold with corners of dimension $n$ 
(see \cite[p.180]{davi07} for the precise definition of a manifold with
corners). 
Then faces are defined and a codimension-one face is called a facet.
We assume that $Q$ is \emph{nice}, which means that every
codimension-$k$ face is a connected component of intersections of $k$
facets. The teardrop, 
which is homeomorphic to the 2-disk, is a
manifold with corners but not nice 
(see \cite[p.181]{davi07}). A simple polytope is a nice manifold with corners and any
intersection of faces is connected unless it is empty. However,
intersections of faces of a nice manifold with corners are not
necessarily connected. For instance, a 2-gon, 
that is the suspension \red{$\lozenge^2$} in the theorem in the Introduction, is
a nice manifold with corners but the intersection of the two facets consists of two vertices. 


Let $S^1$ be the unit circle group of the complex numbers $\C$ and $T$ be an $n$-dimensional connected compact abelian Lie group. As is well-known, $T$ is isomorphic to $(S^1)^n$. We set 
\[
N:=\Hom(S^1,T)\cong\Z^n.
\]
Let $Q$ have $m$ facets and we denote them by $Q_1,\dots,Q_m$. 

\begin{defi}
A function $\v\colon \{Q_1,\dots,Q_m\}\to N$ is called a \emph{characteristic function on $Q$} if it satisfies the following two conditions: 
\begin{enumerate}
\item $\v(Q_i)$ is primitive for each $i\in [m]:=\{1,\dots,m\}$ and 
\item whenever $Q_I=\bigcap_{i\in I}Q_i$ is nonempty for $I\subset [m]$, $\v(Q_i)$'s $(i\in I)$ are linearly independent over $\Q$. 
\end{enumerate}
\red{We denote by $\hat N$ the sublattice of $N$ generated by $v_1,\dots,v_m$.}
\end{defi}
\noindent
We call $\v(Q_i)$'s the \emph{characteristic vectors} and abbreviate $\v(Q_i)$ as $v_i$.
Condition (2) above implies that when $Q$ has a vertex, $\rank\hN=n$. It also implies that when $Q_I\not=\emptyset$, the toral subgroup of $T$ generated by $v_i(S^1)$'s $(i\in I)$, denoted by $T_I$, is of dimension $|I|$ where $|I|$ is the cardinality of $I$. 

To the pair $(Q,\v)$ we associate a quotient space 
\[
X(Q,\v):=(Q\times T)/\!\!\sim
\]
with the equivalence relation $\sim$ on the product $Q\times T$ defined by 
\[
(q,t)\sim (q',t') \text{ if and only if $q=q'$ and $t^{-1}t'\in T_I$}
\]
where $I$ is the subset of $[m]$ such that $Q_I$ is the smallest face of $Q$ containing $q=q'$. The space $X(Q,\v)$ has a $T$-action induced from the natural $T$-action on $Q\times T$. The orbit space of $X(Q,\v)$ by the $T$-action is $Q$ and the quotient map 
\[
\q\colon X(Q,\v)\to Q=X(Q,\v)/T
\]
is induced from the projection map $Q\times T\to Q$. \red{Then it is not difficult to see the following facts (see \cite{ma-so08} for example).} 
A $T$-fixed point in $X(Q,\v)$ corresponds to a vertex of $Q$, so $X(Q,\v)$ has a $T$-fixed point if and only if $Q$ has a vertex. 
If $v_i$'s $(i\in I)$ are a part of a basis of $N$ for every $I$ with $Q_I\not=\emptyset$, then $X(Q,\v)$ is a manifold but otherwise $X(Q,\v)$ is an orbifold. The singularity of $X(Q,\v)$ lies in the union of $\q^{-1}(Q_I)$ over all $I$ with $|I|\ge 2$. 


As mentioned in the Introduction, if $X$ is a compact simplicial toric variety of complex dimension $n$ so that $X$ has an algebraic action of $(\C^*)^n$ having a dense orbit, then the orbit space $Q$ of $X$ by the compact $n$-dimensional subtorus $T$ of $(\C^*)^n$ is a nice manifold with corners and $X$ is homeomorphic to $X(Q,\v)$ where $v_i$'s are primitive edge vectors of the fan associated to $X$. Moreover, faces of $Q$ (even $Q$ itself) are all contractible, which follows from the existence of the residual action of $(\C^*)^n/T$ on $Q=X/T$.

\section{$H^{2n-k}(X(Q,\v))$ 
for $k\le 2$} \label{sect:2}

In this section, we abbreviate $X(Q,\v)$ as $X$ and 
all (co)homology groups will be taken with $\Z$-coefficients unless otherwise stated. When $n=1$, $Q$ is a closed interval if $Q$ has a vertex and a circle otherwise, and $X$ is homeomorphic to $S^2$ or a torus accordingly. We will assume $n\ge 2$ in this section. Remember that $\q\colon X\to Q$ is the quotient map. 


Let $Q^{(n-2)}$ be the union of $Q_I$ over all $I$ with $|I|\ge 2$ and we assume $Q^{(n-2)}\not=\emptyset$. The singular set of $X$ lies in $\q^{-1}(Q^{(n-2)})$ as remarked in Section~\ref{sect:1}. Let $Q'$ be a \lq\lq small closed tubular neighborhood" of $Q^{(n-2)}$ of $Q$ and set $X':=\q^{-1}(Q')$.

\vspace{5mm}
\begin{lemm} \label{lemm:2-1}
$H^{2n-k}(X)\cong H_k(X\backslash \Int X')$ for $k\le 2$.
\end{lemm} 

\begin{proof}
Note that $H^r(X')=0$ for $r\ge 2n-3$ because $X'$ is homotopy equivalent to $\q^{-1}(Q^{(n-2)})$ and $\dim \q^{-1}(Q^{(n-2)})=2n-4$. Therefore, the exact sequence in cohomology for the pair $(X,X')$ yields an isomorphism
\begin{equation} \label{eq:2-1}
H^{2n-k}(X,X')\cong H^{2n-k}(X) \quad\text{for $k\le 2$}. 
\end{equation}

On the other hand, 
\begin{equation} \label{eq:2-2}
\begin{split}
H^{2n-k}&(X,X')\cong H^{2n-k}(X\backslash \Int X',\partial X')\quad \text{by excision}\\
&\cong H_k(X\backslash \Int X') \quad \text{by Poincar\'e-Lefschetz duality}.
\end{split}
\end{equation}
(Note that $X\backslash \Int X'$ is a manifold with boundary $\partial X'$.) The lemma follows from \eqref{eq:2-1} and \eqref{eq:2-2}. 
\end{proof}

\begin{prop} \label{prop:2-1}
$H^{2n}(X)\cong \Z$ and $H^{2n-1}(X)\cong H_1(Q)\oplus N/\hN$. If $H_1(Q_i)=0$ for every $i$, then 
\[
H^{2n-2}(X)\cong\Z^{m-\rank\hN}\oplus H_2(Q)\oplus\big(H_1(Q)\otimes H_1(T)\big)\oplus \big(\wedge^2 N/\hN\wedge N\big).
\]
\end{prop}

\begin{rema}
When $Q$ has a vertex, $\rank\hN=n$ as remarked in Section~\ref{sect:1}. Moreover, when $Q$ has a vertex and $n=2$, the last term $\wedge^2 N/\hN\wedge N$ above is zero. Indeed, since we may assume $N=\Z^2$ and $\hN=\langle e_1, ae_2\rangle$ with some integer $a$, $\hN\wedge N=\langle e_1\wedge e_2\rangle =\wedge^2N$, where $\{e_1,e_2\}$ denotes the standard base of $\Z^2$. 
\end{rema}

\begin{proof}
The statement for $H^{2n}(X)$ follows immediately from Lemma~\ref{lemm:2-1}. 

We shall prove the statement for $H^{2n-1}(X)$. 
Let $Q^0:=(\Int Q)\cap (Q\backslash Q')$ and $Q^1$ be the intersection of $(Q\backslash Q')$ and a small open neighborhood of $\partial Q$ in $Q$.

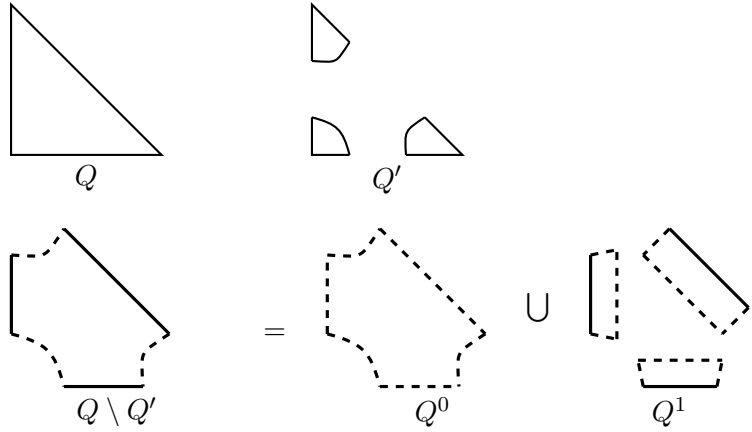
\vspace{7mm}
\begin{tikzpicture}[scale=0.5]

\draw[thick](0,0)--(4,0)--(0,4)-- cycle;
\draw[very thick,,below] (2,0) node{{\large $Q$}};

\coordinate (G) at (8,1);
\coordinate (H) at (9,0);
\draw [thick,bend left,distance=0.7cm] (G) to node [inner sep=0.2pt,circle] {} (H); 
\draw[thick](8,1)--(8,0)--(9,0);

\coordinate (I) at (10.5,0);
\coordinate (J) at (11,1);
\draw [thick,bend left,distance=0.65cm] (I) to node [inner sep=0.2pt,circle] {} (J); 
\draw[thick](10.5,0)--(12,0)--(11,1); 

\coordinate (K) at (8,2.5);
\coordinate (L) at (9,3);
\draw [thick,bend right,distance=0.7cm] (K) to node [inner sep=0.2pt,circle] {} (L); 
\draw[thick](8,2.5)--(8,4)--(9,3);
\draw[very thick,below] (10,0) node{{\large $Q'$}};
\end{tikzpicture}
\vspace{3mm}

\begin{tikzpicture}[scale=0.7]

\coordinate (A) at (0,1);
\coordinate (B) at (1,0);
\draw [dashed,very thick,bend left,distance=0.7cm] (A) to node [inner sep=0.2pt,circle] {} (B); 
\draw[very thick](0,1)--(0,2.5);

\coordinate (C) at (0,2.5);
\coordinate (D) at (1,3);
\draw [dashed,very thick,bend right,distance=0.7cm] (C) to node [inner sep=0.2pt,circle] {} (D); 
\draw[very thick](1,3)--(3,1);

\coordinate (E) at (3,1);
\coordinate (F) at (2.5,0);
\draw [dashed,very thick,bend right,distance=0.7cm] (E) to node [inner sep=0.2pt,circle] {} (F); 
\draw[very thick](1,0)--(2.5,0);

\draw[very thick,below] (2,0) node{{\large $Q\setminus Q'$}};
\draw[very thick] (5,1) node{{\large $=$}}; 

\coordinate (G) at (6,1);
\coordinate (H) at (7,0);
\draw [dashed,very thick,bend left,distance=0.7cm] (G) to node [inner sep=0.2pt,circle] {} (H); 
\draw[dashed,very thick](6,1)--(6,2.5);

\coordinate (I) at (6,2.5);
\coordinate (J) at (7,3);
\draw [dashed,very thick,bend right,distance=0.7cm] (I) to node [inner sep=0.2pt,circle] {} (J); 
\draw[dashed,very thick](7,3)--(9,1);

\coordinate (K) at (9,1);
\coordinate (L) at (8.5,0);
\draw [dashed,very thick,bend right,distance=0.7cm] (K) to node [inner sep=0.2pt,circle] {} (L); 
\draw[dashed,very thick](7,0)--(8.5,0);

\draw[dashed,very thick,below] (8,0) node{{\large $Q^0$}};
\draw[dashed,very thick] (10,1.5) node{{\large $\bigcup$}};

\draw[very thick](11,1)--(11,2.5);
\draw[dashed,very thick](11.5,0.9)--(11.5,2.6);
\draw[dashed,very thick](11,2.5)--(11.5,2.6);
\draw[dashed,very thick](11,1)--(11.5,0.9);

\draw[very thick](12.5,3)--(14,1.5);
\draw[dashed,very thick](12,2.5)--(13.5,1);
\draw[dashed,very thick](12.5,3)--(12,2.5);
\draw[dashed,very thick](14,1.5)--(13.5,1);

\draw[very thick](12,0)--(13.4,0);
\draw[dashed,very thick](11.9,0.5)--(13.5,0.5);
\draw[dashed,very thick](12,0)--(11.9,0.5);
\draw[dashed,very thick](13.4,0)--(13.5,0.5);
\draw[dashed,very thick,below] (12.5,0) node{{\large $Q^1$}}; 
\end{tikzpicture} 

Since 
\[
\begin{split}
&\q^{-1}(Q^0)\simeq Q\times T,\quad \q^{-1}(Q^1)\simeq \bigsqcup_{i=1}^m \big(Q_i\times T/v_i(S^1)\big),\\
& \q^{-1}(Q^0)\cap \q^{-1}(Q^1)\simeq \bigsqcup_{i=1}^m (Q_i\times T),\quad \q^{-1}(Q^0\cup Q^1)=X\backslash X', 
\end{split}
\]
the Mayer-Vietoris exact sequence in homology for the triple $(X\backslash X', \q^{-1}(Q^0), \q^{-1}(Q^1))$ yields the following exact sequence:
\begin{equation} \label{eq:2-3}
\begin{split}
&\bigoplus_{i=1}^m H_2(Q_i\times T)\stackrel{f_2}\longrightarrow H_2(Q\times T)\oplus \bigoplus_{i=1}^m H_2(Q_i\times T/v_i(S^1))\to H_2(X\backslash X') \\
\to &\bigoplus_{i=1}^m H_1(Q_i\times T)\stackrel{f_1}\longrightarrow H_1(Q\times T)\oplus \bigoplus_{i=1}^m H_1(Q_i\times T/v_i(S^1))\to H_1(X\backslash X') \\ 
\to &\bigoplus_{i=1}^m H_0(Q_i\times T)\stackrel{f_0}\longrightarrow H_0(Q\times T)\oplus \bigoplus_{i=1}^m H_0(Q_i\times T/v_i(S^1)).
\end{split}
\end{equation} 
As is easily seen, $f_0$ is injective; so 
\begin{equation} \label{eq:2-4}
H_1(X\backslash X')\cong \coker f_1. 
\end{equation}
We write $f_1$ as $(\psi_1,\varphi_1)$ according to the decomposition of the target space. Since 
\[
\varphi_1\colon \bigoplus_{i=1}^m H_1(Q_i\times T)\to \bigoplus_{i=1}^m H_1(Q_i\times T/v_i(S^1)),
\]
which is $f_1$ composed with the projection on the second factor, is surjective, one has 
\begin{equation} \label{eq:2-5}
\coker f_1\cong H_1(Q\times T)/\psi_1(\ker \varphi_1). 
\end{equation}
Since $H_1(Y\times T)=H_1(Y)\oplus H_1(T)$ for any topological space $Y$, elements in $\ker\varphi_1$ are of the form $(c_1v_1,\dots,c_mv_m)$ with integers $c_i$, where $H_1(T)$ is identified with \red{$N=\Hom(S^1,T)$} in a natural way. It follows that 
\begin{equation} \label{eq:2-6}
H_1(Q\times T)/\psi_1(\ker \varphi_1)\cong H_1(Q)\oplus N/\hN.
\end{equation}
The statement for $H^{2n-1}(X)$ in the proposition follows from \eqref{eq:2-4}, \eqref{eq:2-5}, \eqref{eq:2-6} and Lemma~\ref{lemm:2-1}. 

The computation of $H^{2n-2}(X)$ is similar to that of $H^{2n-1}(X)$. We write $f_2$ as $(\psi_2,\varphi_2)$ similarly to $f_1$. \red{Since $H_1(Q_i)=0$ for any $i$ by assumption}, $\ker f_1$ is a free abelian group of rank $m-\rank \hN$ \red{as is easily seen}; so it follows from \eqref{eq:2-3} that 
\begin{equation} \label{eq:2-7}
H_2(X\backslash X')\cong \Z^{m-\rank \hN}\oplus \coker f_2.
\end{equation} 
Similarly to $\varphi_1$, the map 
\begin{equation} \label{eq:2-8}
\varphi_2\colon \bigoplus_{i=1}^m H_2(Q_i\times T)\to \bigoplus_{i=1}^m H_2(Q_i\times T/v_i(S^1))
\end{equation}
is surjective; so 
\begin{equation} \label{eq:2-9}
\coker f_2\cong H_2(Q\times T)/\psi_2(\ker \varphi_2).
\end{equation}
Here, 
\begin{equation} \label{eq:2-10}
H_2(Y\times T)=H_2(Y)\oplus \big(H_1(Y)\otimes H_1(T)\big)\oplus H_2(T)
\end{equation}
for any topological space $Y$ by the K\"unneth formula. Therefore, since $H_1(Q_i)=0$ by assumption, it follows from \eqref{eq:2-8} and \eqref{eq:2-10} that $\ker\varphi_2$ is contained in $\bigoplus_{i=1}^mH_2(T)$. We note that $H_2(T)$ and $H_2(T/v_i(S^1))$ can be identified with $\wedge^2N$ and $\wedge^2(N/\langle v_i\rangle)$ respectively and the kernel of the projection $\wedge^2 N\to \wedge^2(N/\langle v_i\rangle)$ is $\langle v_i\rangle\wedge N$. Therefore 
\[
\coker f_2\cong 
H_2(Q)\oplus\big(H_1(Q)\otimes H_1(T)\big)\oplus \big(\wedge^2 N/\hN\wedge N\big)
\]
This together with \eqref{eq:2-7} and \eqref{eq:2-9} proves the statement for $H^{2n-2}(X)$ in the proposition. 
\end{proof}

\section{Fundamental groups} \label{sect:2-1}

For a subset $I$ of $[m]$, we define
\[
T^m_I:=\{(h_1,\dots,h_m)\in T^m\mid h_j=1 \quad (\forall j\notin I)\}.
\]
and consider a space 
\[
\mathcal{Z}_Q:=(Q\times T^m)/\!\!\sim_e
\]
where $\sim_e$ is the equivalence relation on the product $Q\times T^m$ defined by 
\[
(q,s)\sim_e (q',s') \text{ if and only if $q=q'$ and $s^{-1}s'\in T^m_I$}
\]
and $I$ is the subset of $[m]$ such that $Q_I$ is the smallest face of $Q$ containing $q=q'$. 

\red{We note that $\mathcal{Z}_Q$ locally admits a smooth structure. Indeed, since $Q$ is a manifold with corners, any point of $Q$ has a neighborhood $U$ homeomorphic to $(\R_{\ge 0})^r\times \R^{n-r}$ for some $0\le r\le n$ and it follows from the construction of $\mathcal{Z}_Q$ that the inverse image of $U$ by the projection map $\kappa\colon \mathcal{Z}_Q\to Q$ is homeomorphic to $\C^r\times \R^{n-r}\times T^{m-r}$. Therefore $\mathcal{Z}_Q$ locally admits a smooth structure and hence is a topological manifold. 
\begin{rema}
When $Q$ is a simple polytope, $\mathcal{Z}_Q$ is called a moment-angle manifold and it is known that $\mathcal{Z}_Q$ admits a smooth structure and is 2-connected (see \cite{bu-pa02} or \cite{bu-pa15}). Moreover, the moment-angle manifold $\mathcal{Z}_Q$ is homotopy equivalent to $\C^{m}-Z$ defined in \cite{co95} (see Theorem 4.7.5 in \cite{bu-pa15}), where $Z$ is the union of coordinate subspaces in $\C^m$ determined by $Q$. 
\end{rema}}

\begin{lemm} \label{lemm:1}
The projection map $\kappa\colon \mathcal{Z}_Q\to Q$ induces an isomorphism $\kappa_*\colon \pi_1(\mathcal{Z}_Q)\cong \pi_1(Q)$ on the fundamental groups. 
\end{lemm}


\begin{proof}
\red{Similarly to the above argument, one can see that $\kappa^{-1}(Q_i)$, where $Q_i$ is a facet of $Q$, is a locally smooth closed manifold. Moreover, it is a locally smooth codimension two submanifold of $\mathcal{Z}_Q$. Indeed, 
a closed tubular neighborhood of $Q_i$ in $Q$ can be identified with $Q_i\times [0,1]$, and $\rho_i\colon \kappa^{-1}(Q_i\times\{1\})\to \kappa^{-1}(Q_i)$, where $\rho_i$ is induced from $((q,1),t)\to (q,t)$ for $q\in Q_i=Q_i\times\{0\}\subset Q_i\times [0,1]\subset Q$ and $t\in T^m$, is a principal $S^1$-bundle, and the total space $E_i$ of the associated complex line bundle can be identified with a closed tubular neighborhood of $Z_i:=\kappa^{-1}(Q_i)$ in $\mathcal{Z}_Q$.} 

\red{Since $Z_i$ is a locally smooth closed codimension two submanifold of $\mathcal{Z}_Q$, the transversality argument can be applied.} 
Therefore, if a continuous map $f\colon S^1\to \mathcal{Z}_Q$ meets $Z_i$, then one can slightly push $f$ in the fiber direction of $E_i$ so that the deformed $f$ does not meet $Z_i$. Applying this deformation to $f$ for every $i$, we see that $f$ is homotopic to a continuous map whose image lies in $\kappa^{-1}(\Int Q)=\Int Q\times T^m$. This means that the inclusion map $\iota\colon \Int Q\times T^m\to \mathcal{Z}_Q$ induces an epimorphism 
\[
\iota_*\colon \pi_1(\Int Q\times T^m)=\pi_1(\Int Q)\times \pi_1(T^m)\to \pi_1(\mathcal{Z}_Q).
\]
Since $\Int Q$ is homotopy equivalent to $Q$, we may replace $\Int Q$ by $Q$ above and we have a sequence 
\begin{equation} 
\pi_1(Q)\times \pi_1(T^m)\stackrel{\iota_*}\longrightarrow \pi_1(\mathcal{Z}_Q) \stackrel{\kappa_*}\longrightarrow \pi_1(Q),
\end{equation}
where the composition $\kappa_*\circ\iota_*$ agrees with the projection on the first factor, so that the kernel of $\iota_*$ is contained in the second factor $\pi_1(T^m)$. 

Let $S_i$ be the $i$-th $S^1$-factor of $T^m$ and choose a point $q_i\in (Q_i\times\{1\})\cap\Int Q$. Then $\iota(\{q_i\}\times S_i)$ is a fiber of the principal $S^1$-bundle $\rho_i\colon \kappa^{-1}(Q_i\times \{1\})\to Z_i=\kappa^{-1}(Q_i)$, so it shrinks to a point in $Z_i$. 
Therefore $\pi_1(T^m)$ is in the kernel of the epimorphism $\iota_*$ and this implies the lemma.
\end{proof}

We recall a result from Bredon's book \cite{bred72}. 

\begin{lemm}\cite[Corollary 6.3 \red{on} p.91]{bred72}. \label{lemm:2}
If $X$ is \red{an} arcwise connected $G$-space, $G$ compact Lie, and if there is an orbit which is connected (e.g., $G$ connected or $X^G\not=\emptyset$), then the quotient map $X\to X/G$ induces an epimorphism on their fundamental groups.
\end{lemm}

The characteristic map $\v\colon \{Q_1,\dots,Q_m\}\to \Hom(S^1,T)$ defines a homomorphism 
$T^m\to T$, denoted $\v$ again. Note that $\v(T^m)$ is a subtorus of $T$ of dimension $\rank\hN$, in particular, $\v$ is surjective if and only if $\rank \hN=\rank N$ (this is the case when $Q$ has a vertex). 
The product map $id\times \v\colon Q\times T^m\to Q\times T$ induces a continuous map 
\red{$${V}\colon \mathcal{Z}_Q=Q\times T^m/\sim_e\to Q\times T/\sim=X(Q,\v)=X$$
and it further induces an injective continuous map
$$ \bar{V} \colon \mathcal{Z}_Q/\ker\v\to X,$$
so that $\bar{V}$ is a homeomorphism if $\v$ is surjective since the spaces are compact and Hausdorff.}
\begin{prop} \label{prop:2-2}
If $Q$ has a vertex, then $\pi_*\colon \pi_1(X)\cong\pi_1(Q)$. 
\end{prop}

\begin{proof}
We have a sequence
\[
\kappa_*=\pi_*\circ {V}_*\colon \pi_1(\mathcal{Z}_Q)\stackrel{{{V}}_*}\longrightarrow \pi_1(X) \stackrel{\pi_*}\longrightarrow \pi_1(Q).
\]
Since $\kappa_*$ is an isomorphism by Lemma~\ref{lemm:1}, it suffices to prove that ${V}_*$ is surjective. 

Since $Q$ has a vertex, $\rank \hN=\rank N$ and the homomorphism $\v\colon T^m\to T$ is surjective; \red{so the map $\bar{V}\colon \mathcal{Z}_Q/\ker\v\to X$ above is a homeomorphism.} 
Since $\hN$ is a sublattice of $N$ of finite index, there is a finite covering homomorphism $\rho\colon \hat T\to T$ corresponding to $\hN$, where $\hat T$ is also a compact connected abelian Lie group of dimension $n$ (precisely speaking, $\rho_*(\pi_1(\hat T))=\hN$ when $N$ is regarded as $\pi_1(T)$) and the characteristic function $v$ uniquely determines a characteristic function $\hat v\colon \{Q_1,\dots,Q_m\}\to \Hom(S^1,\hat T)$ such that $\rho_*(\hat v(Q_i))=\v(Q_i)$ for any $i$. Then we have 
\[
\hat X:=X(Q,\hat\v)=(Q\times\hat T)/\!\!\sim
\]
and $\hat \v$ induces a homomorphism $T^m\to \hat T$, denoted $\hat\v$
again similarly to $\v$, and $\hat X=\mathcal{Z}_Q/\ker\hat\v$.
Moreover, we have $X=\hat X/\ker\rho$. Namely, the quotient map
$V\colon \mathcal{Z}_Q\to X$ factors as the composition of two quotient maps 
\[
\mathcal{Z}_Q\stackrel{\alpha}\longrightarrow \mathcal{Z}_Q/\ker\hat\v=\hat X\stackrel{\beta}\longrightarrow \hat X/\ker\rho=X.
\]
\red{The} 
Theorem on Elementary Divisors (see Section~\ref{sect:4}) implies that
since $\hat\v(Q_i)$'s span $\hN$, the homomorphism $\hat\v\colon T^m\to
\hat T$ composed with a suitable automorphism of $T^m$ can be viewed as
a projection map if we take a suitable identification of $\hat T$ with
$T^n$; so $\ker\hat\v$ is connected and hence $\alpha_*\colon \pi_1(\mathcal{Z}_Q)\to \pi_1(\hat X)$ is surjective by Lemma~\ref{lemm:2}. The action of $\hat T$ on $\hat X$ has a fixed point since $Q$ has a vertex and $\ker\rho$ is contained in $\hat T$, so the action of $\ker\rho$ on $\hat X$ has a fixed point. Therefore $\beta_*\colon \pi_1(\hat X)\to \pi_1(X)$ is also surjective again by Lemma~\ref{lemm:2}. 
\end{proof}

\red{
\begin{rema}
As mentioned in the Introduction, even if $Q$ is a simple polytope, $X=\mathcal{Z}_Q/\ker\v$ is not necessarily a compact toric orbifold because the characteristic map $\v$ is not necessarily coming from primitive vectors of a complete simplicial fan.
\end{rema}} 


\begin{coro} \label{coro:2-1}
If $Q$ has a vertex and $H_1(Q)=H_2(Q)=0$, then $H^1(X)=0$ and $H^2(X)\cong \Z^{m-n}$. 
\end{coro}

\begin{proof}
By Proposition~\ref{prop:2-2}, $\pi_1(X)\cong\pi_1(Q)$ and hence $H_1(X)\cong H_1(Q)$. Therefore $H_1(X)=0$ since $H_1(Q)=0$ by assumption and hence $H^1(X)=0$ and $H^2(X)$ has no torsion by the universal coefficient theorem. On the other hand, since $X$ is an orbifold, Poincar\'e duality holds with $\Q$-coefficients. Therefore the rank of $H^2(X)$ is equal to that of $H^{2n-2}(X)$, that is $m-n$ by Proposition~\ref{prop:2-1} and its subsequent remark. 
\end{proof}

\section{Low dimensional cases} \label{sect:2-2}

A nice manifold with corners $Q$ is called \emph{face-acyclic} (\cite{ma-pa06}) if every face of $Q$ (even $Q$ itself) is acyclic. 
\red{We note that if $Q$ is face-acyclic, then $Q$ must have a vertex. Indeed, let $F$ be a face of $Q$ of minimum dimension. Then $F$ has no boundary because the boundary of $F$ must consist of faces of smaller dimensions, so $F$ is a closed manifold. But since $F$ is acyclic, this means that $F$ is a point. Therefore $Q$ has a vertex.}

We shall apply the previous results when $Q$ is face-acyclic and $n=\dim Q$ is $2$ or $3$. The following corollary follows from Proposition~\ref{prop:2-1} and Corollary~\ref{coro:2-1}. 

\begin{coro} \label{coro:2-2}
Suppose that $Q$ is face-acyclic and $\dim Q=2$, that is, $Q$ is an $m$-gon $(m\ge 2)$. Then we have 
\[
H^j(X)\cong \begin{cases} \Z \quad&(j=0,4)\\
\Z^{m-2} \quad &(j=2)\\
N/\hat N \quad &(j=3)\\
0\quad&\text{\rm{(otherwise)}}.
\end{cases}
\]
\end{coro} 

\begin{exam}
Let $a$ be a positive integer. Take $Q$ to be a 2-simplex, $N=\Z^2$ and 
\[
v_1=(2a,1),\ v_2=(0,1),\ v_3=(-a,-1).
\]
Then $\hat N=\langle ae_1,e_2\rangle$ and $N/\hat N\cong\Z/a$. 
The space $X$ is not a weighted projective space \brown{when} $a\ge 2$ since it has torsion in cohomology, where $\{e_1,e_2\}$ denotes the standard base of $\Z^2$ as before. 
\end{exam}

\begin{coro} \label{coro:2-3}
Suppose that $Q$ is face-acyclic and $\dim Q=3$. Then 
\[
H^j(X)\cong \begin{cases} \Z \quad&(j=0,6)\\
\Z^{m-3} \quad &(j=2)\\
0 \text{ or some torsion group} \quad &(j=3)\\
\Z^{m-3}\oplus \wedge^2N/(\hat N\wedge N)\quad &(j=4)\\
N/\hat N \quad &(j=5)\\
0\quad&\text{\rm{(otherwise)}}.
\end{cases}
\]
\end{coro} 

\begin{proof}
Since $Q$ is face-acyclic, 
\red{$Q$ has a vertex as remarked at the beginning of this section;} so all the statements except for $j=3$ follows from Proposition~\ref{prop:2-1} and Corollary~\ref{coro:2-1}. In order to prove the statement for $j=3$, it suffices to show $H^3(X;\Q)=0$ and this is equivalent to showing that the euler characteristic of $X$ is $2m-4$ (note that we know the rank of $H^j(X)$ except for $j=3$). \\
\indent
Since $Q$ is face-acyclic and of dimension 3, the boundary of $Q$ is a 2-sphere, every 2-face of $Q$ is a 2-disk and the number of 2-faces is $m$ by definition. Let $V$ be the number of vertices of $Q$. Then the number of edges of $Q$ is $3V/2$ and hence we obtain an identity $V-3V/2+m=2$ by Euler's formula, which implies $V=2m-4$. On the other hand, it is known that the euler characteristic of $X$ is equal to that of \brown{the} $T$-fixed point set $X^T$ (see \cite[Theorem 10.9 in p.163]{bred72}). In our case $X^T$ is isolated and corresponds to the vertices of $Q$. Therefore, the euler characteristic of $X$ is \brown{equal to} $V$, that is $2m-4$. 
\end{proof}

\begin{exam}
It happens that $\hN\wedge N=\wedge^2 N$ even if $\hN\not=N$. For instance, take $Q$ to be a 3-simplex, $N=\Z^3$ and 
\[
v_1=(0,0,1),\ v_2=(2,0,1),\ v_3=(0,1,1),\ v_4=(-2,-1,-1).
\]
Then 
\[
\hN=\langle 2e_1,e_2,e_3\rangle,\quad \hN\wedge N=\langle e_1\wedge e_2, e_1\wedge e_3, e_2\wedge e_3\rangle=\wedge^2 N,
\]
where $\{e_1,e_2,e_3\}$ denotes the standard base of $\Z^3$. 
\end{exam}

Corollary~\ref{coro:2-3} says that if $\hN=N$, then $H^j(X)$ has no torsion except $j=3$. However, $H^3(X)$ can be nontrivial (so, a nontrivial torsion group) \brown{when} $\hN=N$. We shall give such an example below. One can also find many such examples using \brown{Maple package} torhom. 
\begin{exam}
Let $a$ be a positive integer and take the following five primitive vectors in $\Z^3$:
\[
\begin{split}
&v_+=(0,0,1),\\ 
&v_1=(2a,1,0),\ v_2=(0,1,0),\ v_3=(-a,-1,0),\\
& v_-=(1,0,-1).
\end{split}
\]
Then $\hN=N$. We consider the complete simplicial fan $\Delta$ having the following six 3-dimensional cones 
\[
\angle v_+v_1v_2,\ \angle v_+v_1v_3,\ \angle v_+v_2v_3,\ \angle v_-v_1v_2,\ \angle v_-v_1v_3,\ \angle v_-v_2v_3
\]
where $\angle v_\epsilon v_iv_j$ ($\epsilon\in\{+,-\}$, $i,j\in\{1,2,3\}$) denotes the cone spanned by $v_\epsilon,v_i$ and $v_j$. Let $X$ be the compact simplicial toric variety associated to the fan $\Delta$. Let $\rho$ be the projection of $\R^3$ on the line $\R$ corresponding to the last coordinates of $\R^3$. Then the vectors $\v_1,\v_2,\v_3$ are in the kernel of $\rho$ and $\rho(\v_\pm)$ are primitive vectors and \brown{determine} the complete 1-dimensional fan. This means that we have a fibration $F\to X\to \C P^1$ where the fiber $F$ is the compact simplicial toric variety associated to the fan obtained by projecting the fan $\Delta$ on the plane $\R^2$ corresponding to the first two coordinates of $\R^3$. The $E_2$-terms of the Serre spectral sequence of the fibration are 
\[
E_2^{p,q}=H^p(\C P^1;H^q(F))
\]
and $E_2^{p,q}=0$ unless $p=0,2$ and $q=0,2,3,4$ by Corollary~\ref{coro:2-2}. Therefore all the differentials except 
\[
d_2^{0,3}\colon E_2^{0,3}\to E_2^{2,2}\quad{\text{and}}\quad d_2^{0,4}\colon E_2^{0,4}\to E_2^{2,3}
\]
are trivial. Here, $E_2^{0,3}=H^0(\C P^1;H^3(F))=H^3(F)$ is trivial or a torsion group by Corollary~\ref{coro:2-2} while $E_2^{2,2}=H^2(\C P^1;H^2(F))=H^2(F)$ is a free abelian group again by Corollary~\ref{coro:2-2}, so $d_2^{0,3}$ must be trivial. Therefore $E_2^{0,3}=E_\infty^{0,3}$. Since $E_2^{p,q}$ with $p+q=3$ vanishes unless $(p,q)=(0,3)$, we obtain an isomorphism $H^3(X)\cong H^3(F)$. Here $H^3(F)\cong\Z/a$ again by Corollary~\ref{coro:2-2} (see Example after Corollary~\ref{coro:2-2}) and hence we have $H^3(X)\cong\Z/a$. On the other hand, since $\hN=N$ as remarked above, $H^j(X)$ has no torsion for $j\not=3$ by Corollary~\ref{coro:2-3}. 
\end{exam}

\section{A necessary condition for no $p$-torsion} \label{sect:3}

Let $I$ be a subset of $[m]$ with $Q_I\not=\emptyset$. Although $Q_I$ is not necessarily connected, we understand that $Q_I$ stands for a connected component of $Q_I$ in this section for notational convenience. Then the characteristic function $\v$ associates a characteristic function $\v_I$ on $Q_I$ as follows. Since $v_i$'s $(i\in I)$ are linearly independent over $\Q$, they span a $|I|$-dimensional linear subspace of $N\otimes \R$ and its intersection with $N$ is a rank $|I|$ sublattice of $N$, denoted $N_I$. Then $N(I):=N/N_I$ is a free abelian group of rank $n-|I|$ and we denote the projection map from $N$ to $N(I)$ by $\pi_I$. \blue{If $Q_I\cap Q_j$ is nonempty for $j\in [m]\backslash I$,} then its connected components are facets of $Q_I$, and any facet of $Q_I$ is of this form. The element $\pi_I(v_j)\in N(I)$ is not necessarily primitive and we define $\v_I(Q_I\cap Q_j)$ to be the primitive vector in $N(I)$ which has the same direction as $\pi_I(v_j)$, where $Q_I\cap Q_j$ also stands for a connected component of $Q_I\cap Q_j$. Then one can see that $\v_I$ is a characteristic function on $Q_I$. Similarly to $\hN$, one can define a sublattice $\hN(I)$ of $N(I)$ using $\v_I$. We allow $I=\emptyset$ and understand $Q_\emptyset=Q$, $N(\emptyset)=N$ and $\hN(\emptyset)=\hN$. We define 
\[ 
\m(Q_I):=
\begin{cases} 
|N(I)/\hN(I)|\quad&\text{when $Q_I\not=\emptyset$},\\
1\quad&\text{when $Q_I=\emptyset$}.
\end{cases}
\] 
Here $|N(I)/\hN(I)|$ is not necessarily finite. For instance, take
$Q=S^1\times [-1,1]$ and assign characteristic vectors $(1,0)$ and
$(-1,0)$ to the facets $S^1\times \{1\}$ and $S^1\times \{-1\}$
respectively. Then $N/\hN$ is an infinite cyclic group and hence
$|N(I)/\hN(I)|$ is infinite for $I=\emptyset$. One can easily construct
a similar example such that $|N(I)/\hN(I)|$ is infinite for some
$I\not=\emptyset$.

\begin{rema}
When $|I|=n$, $N(I)=\{0\}$; so $\m(Q_I)=1$. When $|I|=n-1$, $N(I)$ is
of rank one and $\hN(I)$ is generated by a primitive vector; so
$\hN(I)=N(I)$ and hence $\m(Q_I)=1$ in this case too. Another case
which ensures $\m(Q_I)=1$ is the following. Let $q$ be a vertex of
$Q$. Then there is a subset $J$ of $[m]$ with $|J|=n$ such that $q\in Q_J$. If $\{v_j\}_{j\in J}$ is a base of $N$, then $\mu(Q_I)=1$ for every subset $I$ of $J$, which easily follows from the definition of $\m(Q_I)$. 
\end{rema}

We note that for a prime number $p$, $H^*(X(Q,\v);\Z)$ has no $p$-torsion if and only if $H^{odd}(X(Q,\v);\Z/p)=0$, which follows from the universal coefficient theorem (see \cite[Corollary 56.4]{munk84}). 

\red{\begin{lemm}\label{lemm:3-0}\cite[Theorem 2.2 on pp.376-377]{bred72}. 
Let a group $G$ of prime order $p$ act on a finite dimensional space $X$ with $A\subset X$ closed and invariant. Suppose that $G$ acts trivially on $H^*(X,A;\Z)$. Then 
\[\sum_{i\geq 0}{\rm rk}~H^{k+2i}(X^G,X^G\cap A;\Z/p)\leq\sum_{i\geq0}{\rm rk}~H^{k+2i}(X,A;\Z/p).\]
\end{lemm}}
\begin{prop} \label{prop:3-1}
If $H^{odd}(X(Q,\v);\Z/p)=0$, then $H_1(Q_I;\Z/p)=0$ and $\m(Q_I)$ is finite and coprime to $p$ for every $I$.
\end{prop}

\begin{proof}
We abbreviate $X(Q,\v)$ as $X$ as before. Since $H^{odd}(X;\Z/p)=0$,
we have $H^{odd}(X^{G};\Z/p)=0$ for every $p$-subgroup $G$ of $T_I$ by
repeated use of \red{Lemma~\ref{lemm:3-0}. In fact, let $G$ be an order $p$ subgroup of $S^1$. The induced action of $G$ on $H^*(X)$ is trivial because $G$ is contained in the connected group $S^1$. Then ${\rm rk}~H^{odd}(X^G; \Z/p)\leq {\rm rk}~H^{odd}(X; \Z/p)$ by Lemma~\ref{lemm:3-0} applied with $A=\emptyset$. Therefore, $H^{odd}(X^G; \Z/p)=0$ by assumption. Repeating the same argument for $X^G$ with the induced action of $S^1/G$, which is again a circle group, we conclude that $H^{odd}(X^G; \Z/p)=0$ for any $p$-subgroup $G$ of $S^1$.} 

For a
positive integer $k$, let $G_k$ be the $p$-subgroup of $T_I$ consisting
of all elements of order at most $p^k$. Then $G_k\subset G_{k'}$ for
$k\le k'$ and the union $\bigcup_{k=1}^\infty G_k$ is dense in $T_I$.
Therefore $X^{G_k}=X^{T_I}$ if $k$ is sufficiently
large.\footnote{Detailed explanation about this assertion. Since the
set of isotropy groups of $X$ is finite, there is a positive integer
$r$ such that $X^{G_k}=X^{G_r}$ for every $k\ge r$. Since $G_r$ is a
subgroup of $T_I$, we have $X^{G_r}\supset X^{T_I}$. We shall prove
the opposite inclusion. Let $x\in X^{G_r}$. The isotropy subgroup
$T_x$ at $x$ contains $G_k$ for every $k\ge r$ because
$X^{G_k}=X^{G_r}$ but since $T_x$ is a closed subgroup of $T$, $T_x$
must contain the closure of $\cup_{k=r}^\infty G_k$, that is $T_I$.
Therefore $x\in X^{T_I}$ and hence $X^{G_r}=X^{T_I}$.} Since $X_I=\q^{-1}(Q_I)$ is a connected component of $X^{T_I}$, this shows that $H^{odd}(X_I;\Z/p)=0$. But $H^{2(n-|I|)-1}(X_I)$ is isomorphic to $H_1(Q_I)\oplus N(I)/\hN(I)$ by Proposition~\ref{prop:2-1} and hence the universal coefficient theorem implies the proposition. 
\end{proof}

When $H^{odd}(X(Q,\v);\Z/p)=0$, Proposition~\ref{prop:3-1} gives a constraint on the topology of $Q_I$, that is $H_1(Q_I;\Z/p)=0$. It is proved in \cite{ma-pa06} that if $X(Q,\v)$ is a manifold and $H^{odd}(X(Q,\v);\Z)=0$, then $Q$ is face-acyclic. This implies that there will be more constraints on the topology of $Q_I$ when $H^{odd}(X(Q,\v);\Z/p)=0$, to be more precise, we expect that $Q$ is \emph{face $p$-acyclic} which means that (every component of) $Q_I$ is acyclic with $\Z/p$-coefficients for every $I$. Therefore, in order to consider the converse of Proposition~\ref{prop:3-1}, it would be appropriate to assume that $Q$ is face $p$-acyclic. We will prove in Section~\ref{sect:5} that the converse holds in some cases while we will see in Section~\ref{sect:6} that the converse does not hold in general.

\section{Theorem on Elementary Divisors} \label{sect:4}

\red{We recall the Theorem on Elementary Divisors and deduce two facts from it, which will play a role in the next section.} 

\begin{theo}[Theorem on Elementary Divisors, see \cite{wa67}] \label{theo:4-1}
Let $N'$ be a submodule of rank $n'$ in $N=\Z^n$. Then there are bases $\{u'_1,\dots,u'_{n'}\}$ of $N'$ and $\{u_1,\dots,u_n\}$ of $N$ such that $u'_i=\epsilon_i u_i$ with some integer $\epsilon_i$ for $i=1,2,\dots,n'$ and $\epsilon_1|\epsilon_2|\dots|\epsilon_{n'}$. Moreover if $A=(a_1,\dots,a_k)$ is an $n\times k$ integer matrix whose column vectors $a_1,\dots,a_k$ generate $N'$ and 
\[
\delta_i:=\gcd\{\det B\mid B \text{ is an $i\times i$ submatrix of $A$}\},
\]
then $\delta_i=\delta_{i-1}\epsilon_i$ for $i=1,2,\dots,n'$. In particular, if $n'=n$, then $\delta_n=|N/N'|$. 
\end{theo}

We deduce two facts from Theorem~\ref{theo:4-1}.

\begin{lemm} \label{lemm:4-1}
Let $A$ be an $n\times n$ integer matrix of rank $n$ and $\tilde A\colon \R^n/\Z^n\to \R^n/\Z^n$ be the epimorphism induced from $A$. Then $\ker\tilde A\cong \coker A$. 
\end{lemm}

\begin{proof}
By Theorem~\ref{theo:4-1} we may think of $A$ as the diagonal matrix with diagonal entries $\epsilon_1,\dots,\epsilon_n$. Then one easily sees that $\ker\tilde A$ and $\coker A$ are both isomorphic to $\prod_{i=1}^n\Z/\epsilon_i$, proving the lemma. 
\end{proof}

Let $a_1,\dots,a_{n+1}$ be elements of $\Z^n$ which generate a sublattice $\langle a_1,\dots,a_{n+1}\rangle$ of rank $n$ and set \red{$\d_i:=|\det((a_j)_{j\not=i})|$} for $i\in [n+1]$. It follows from Theorem~\ref{theo:4-1} that 
\begin{equation} \label{eq:4-1}
\delta_n=\gcd(\d_1,\dots,\d_{n+1})=|\Z^n/\langle a_1,\dots,a_{n+1}\rangle|.
\end{equation}
Suppose that $a_{n+1}$ is primitive. Let $\bar{a}_k$ $(k\not=n+1)$ be 
the projection image of $a_k$ on $\Z^{n}/\langle a_{n+1}\rangle$ and let $a'_k$ be the primitive vector in the quotient lattice $\Z^{n}/\langle a_{n+1}\rangle$ which has the same direction as $\bar{a}_k$ when $\bar{a}_k$ is nonzero and $a'_k$ be the zero vector when so is $\bar{a}_k$. Set $\d_j':=\det(a_1',\dots,\widehat{a_j'},\dots,a_n')$
. With this understood we have the following. 

\begin{lemm} \label{lemm:4-2}
$\gcd(\d_1,\dots,\d_{n})\big|\d_{n+1}$, i.e., $\gcd(\d_1,\dots,\d_{n})=\gcd(\d_1,\dots,\d_{n+1})$. Moreover, $\gcd(\d_1',\dots,\d_{n}')\big|\gcd(\d_1,\dots,\d_{n+1})$.
\end{lemm}

\begin{proof}
\red{Theorem~\ref{theo:4-1} applied with $N'$ generated by $a_{n+1}$ says that there is a basis $\{u_1,\dots,u_n\}$ of $N=\Z^n$ such that $a_{n+1}=\epsilon_1 u_1$ with some integer $\epsilon_1$. But since $a_{n+1}$ is primitive, we have $\epsilon_1=\pm 1$. Therefore,} 
we may assume that $a_{n+1}=(0,\dots,0,1)^T$ \red{through a linear transformation of $\Z^n$}. We have 
\begin{equation} \label{eq:4-2}
\d_{n+1}=|\det(a_1,\dots,a_n)|=\big|\sum_{j=1}^n a_j^n\tilde{a}_j^n\big|
\end{equation}
where $a_j^n$ is the $(n,j)$ entry of the matrix $(a_1,\dots,a_n)$ and $\tilde{a}_j^n$ is its cofactor. Since $a_{n+1}=(0,\dots,0,1)^T$, $\tilde{a}_j^n$ agrees with $d_j=|\det(a_1,\dots,\widehat{a_j},\dots,a_{n+1})|$ up to sign. Therefore $\tilde{a}_j^n$ is divisible by $\gcd(d_1,\dots,d_n)$ for every $j$ and this together with \eqref{eq:4-2} implies the former statement in the lemma.

Since $a_{n+1}=(0,\dots,0,1)^T$, $\Z^n/\langle a_{n+1}\rangle$ can naturally be identified with $\Z^{n-1}$ and we have 
\begin{equation} \label{eq:4-3}
d_j=|\det(a_1,\dots,\widehat{a_j},\dots,a_{n+1})|=|\det(\bar{a}_1,\dots,\widehat{\bar{a}_j},\dots,\bar{a}_{n})|\quad \text{for $j=1,2,\dots,n$}
\end{equation}
where $\bar{a}_k$ $(k=1,2,\dots,n)$ is the projection image of $a_k$ on $\Z^n/\langle a_{n+1}\rangle=\Z^{n-1}$. Since $\bar{a}_k$ is a positive scalar multiple of $a'_k$, $\d_j'=|\det(a_1',\dots,\widehat{a_j'},\dots,a_n')|$ divides the latter term \red{in \eqref{eq:4-3} above and hence $d_j$}. This together with the former statement in the lemma implies the latter statement in the lemma. 
\end{proof}

\section{Converse of Proposition~\ref{prop:3-1} in three cases} \label{sect:5}

In this section we show that if $Q$ is face $p$-acyclic and has the same face poset as one of the following:
\begin{description}
\item[Case 1] the suspension \red{$\lozenge^n$} of an $(n-1)$-simplex $\Delta^{n-1}$ (see the Introduction), 
\item[Case 2] the $n$-simplex $\Delta^n$,
\item[Case 3] the prism $\Delta^{n-1}\times [-1,1]$,
\end{description}
then the converse of Proposition~\ref{prop:3-1} holds, \red{i.e. if $\mu(Q_I)$ is finite and coprime to $p$ for every $I$, then $H^{odd}(X(Q,\v);\Z/p)=0$.}

\red{
First we establish Case 1.  Then we reduce Case 2 to Case 1 by collapsing a face of $Q$ to a point.  In Case 3,  according to the characteristic function $\v$, we collapse one or two faces of $Q$ to a point reducing Case 3 to Case 2 or Case 1.  The argument then becomes much more complicated than that reducing Case 2 to Case 1.  It would be interesting to see whether this inductive argument works for an arbitrary product of simplices. 
}

Let $q$ be a vertex of $Q$. Then $q$ lies in $Q_I$ for some $I\subset [m]$ with $|I|=n$. We set
\red{\[
\d_Q(q):=|\det((v_i)_{i\in I})|
\]}
where $v_i=v(Q_i)$ as before. 

\medskip
{\bf Case 1.} In this case $Q$ has two vertices, say $q$ and $q'$, and $\d_Q(q)=\d_Q(q')=\m(Q)$. 

\begin{prop} \label{prop:5-1}
Suppose that $Q$ is face $p$-acyclic, has the \red{same} face poset as \red{$\lozenge^n$} and $\m(Q)$ is coprime to $p$. Then $X(Q,\v)$ has the same cohomology as $S^{2n}$ with $\Z/p$-coefficients, in particular $H^{odd}(X(Q,\v);\Z/p)=0$. 
\end{prop}

\begin{proof}
\red{When $n=1$, $Q$ is a closed interval and $X(Q,\v)$ is homeomorphic to $S^2$; so the proposition holds when $n=1$. In the following we assume $n\ge 2$, so that $Q$ has $n$ facets.} 

Let $T^n=(S^1)^n$. Then $\Hom(S^1,T^n)$ is naturally isomorphic to $\Z^n$ and we identify them. Let $\{e_i\}_{i=1}^n$ be the standard basis of $\Z^n$ and $e\colon \{Q_1,\dots,Q_n\}\to \Z^n=\Hom(S^1,T^n)$ be the characteristic function assigning $e_i$ to $Q_i$. Then we have a $T^n$-space $X(Q,e)$ which is actually a manifold because $\{e_i\}_{i=1}^n$ is a basis of $\Z^n$. 

The characteristic vectors $v_i\in N=\Hom(S^1,T)$ define an epimorphism $\tv\colon T^n \to T$ sending $(h_1,\dots,h_n)$ to $\prod_{i=1}^nv_i(h_i)$. One can see that the surjective map from $Q\times T^n$ to $Q\times T$ sending $(q,t)$ to $(q,\tv(t))$ descends to a $\tv$-equivariant map from $X(Q,e)$ to $X(Q,v)$ and further descends to a homeomorphism 
$$X(Q,e)/\ker\tv\approx X(Q,\v).$$ 
Here $|\ker\tv|=|N/\hN|$ by Lemma~\ref{lemm:4-1} and it is coprime to $p$ by assumption. Moreover, since $\ker\tv$ is a subgroup of the connected group $T^n$ acting on $X(Q,e)$, the induced action of $\ker\tv$ on $H^*(X(Q,e);\Z/p)$ is trivial. Therefore we have 
$$H^*(X(Q,e)/\ker\tv;\Z/p)\cong H^*(X(Q,e);\Z/p)$$ 
(see \cite[Theorem 2.4 in p.120]{bred72}) and hence it suffices to prove that $X(Q,e)$ has the same cohomology as $S^{2n}$ with $\Z/p$-coefficients. 

Since $Q$ has the same face poset as \red{$\lozenge^n$} and every face of \red{$\lozenge^n$} is contractible, there is a face preserving map $f\colon Q\to \red{\lozenge^n}$ which induces an isomorphism on the face posets. Since $Q$ is face $p$-acyclic, $f$ induces an isomorphism on cohomology with $\Z/p$-coefficients at each face. Similarly to the definition of $e$, one has a characteristic function on \red{$\lozenge^n$}, also denoted by $e$. Then the map from $Q\times T^n$ to $\red{\lozenge^n}\times T^n$ sending $(q,t)$ to $(f(q),t)$ descends to a map $$X(Q,e)\to X(\red{\lozenge^n},e)$$ which induces an isomorphism on cohomology with $\Z/p$-coefficients. Since $X(\red{\lozenge^n},e)$ is homeomorphic to $S^{2n}$, this proves the desired result. 
\end{proof}

\medskip
{\bf Case 2.} Since $Q$ has the same face poset as the $n$-simplex $\Delta^n$, $Q$ has $n+1$ facets $Q_1,\dots,Q_{n+1}$ and $n+1$ vertices 
$q_1,\dots,q_{n+1}$. We number them in such a way that $q_i$ is the unique vertex not contained in $Q_i$. 
It follows from \eqref{eq:4-1} and Lemma~\ref{lemm:4-2} that 
\red{
\begin{equation} \label{eq:5-1}
\begin{split}
&\m(Q)=\gcd(\d_Q(q_1),\dots,\d_Q(q_{n+1}))=\gcd(\d_Q(q_1),\dots,\widehat{\d_Q(q_i)},\dots,\d_Q(q_{n+1})) ~\text{and} \\
&\m(Q_i)~\text{divides}~ \m(Q) \quad\text{for any $i\in [n+1]$}.
\end{split}
\end{equation}
In fact, the former identity in \eqref{eq:5-1} follows from \eqref{eq:4-1}. The latter identity with $i=n+1$ follows from Lemma~\ref{lemm:4-2} but the same proof of Lemma~\ref{lemm:4-2} works for any $i$ and proves the desired identity. Similarly, the last assertion in \eqref{eq:5-1} also follows from (the proof of) Lemma~\ref{lemm:4-2}. 
}

\begin{prop} \label{prop:5-2}
Suppose that $Q$ is face $p$-acyclic, has the same face poset as $\Delta^n$ and $\m(Q)$ is coprime to $p$. Then $H^{odd}(X(Q,\v);\Z/p)=0$. 
\end{prop}

\begin{proof}
We abbreviate $X(Q,\v)$ as $X$. We prove the proposition by induction on $n$. When $n=1$, $Q$ is a closed interval and $X$ is homeomorphic to $S^2$; so the proposition holds in this case. We assume that the proposition holds for any face $p$-acyclic $(n-1)$-dimensional manifold with corners satisfying the assumption in the proposition. For every $i$, $Q_i$ has the same face poset as $\Delta^{n-1}$ and $\m(Q_i)|\m(Q)$ by \eqref{eq:5-1}, so $H^{odd}(X_i;\Z/p)=0$ by the induction assumption, \red{where $X_i=\pi^{-1}(Q_i)$ and $\pi: X\rightarrow Q$ is the quotient map.} On the other hand, since $\m(Q)=\gcd(\d_Q(q_1),\dots,\d_Q(q_{n+1}))$ is coprime to $p$ by assumption, $\d_Q(q_i)$ is coprime to $p$ for some $i$. For such $i$, $Q/Q_i$ is face $p$-acyclic, has the same face poset as \red{$\lozenge^n$} and $\m(Q/Q_i)=\d_Q(q_i)$ is coprime to $p$, so $H^{odd}(X/X_i;\Z/p)=0$ by Proposition~\ref{prop:5-1}. These together with the exact sequence 
\[
\to H^{odd}(X/X_i;\Z/p)\to H^{odd}(X;\Z/p)\to H^{odd}(X_i;\Z/p)\to
\]
show $H^{odd}(X;\Z/p)=0$. 
\end{proof}

\medskip
{\bf Case 3.} We denote the facets of $Q$ corresponding to $\Delta^{n-1}\times\{\pm 1\}$ by $Q_{\pm}$ and the others by $Q_1,\dots,Q_n$. Accordingly, we abbreviate the characteristic vectors $\v(Q_\pm)$ as $\v_\pm$ and $\v(Q_i)$ as $\v_i$. We denote the vertices in $Q_\epsilon$ by $q_1^\epsilon,\dots,q_n^\epsilon$ for $\epsilon=\pm$ in such a way that $q_i^\epsilon$ is not contained in $Q_i$.

\begin{lemm} \label{lemm:5-1}
\red{Suppose that $Q$ is face $p$-acyclic and has the same face poset as $\Delta^{n-1}\times [-1,1]$.}
If $\m(Q)$ is coprime to $p$ and either $\m(Q_+)$ or $\m(Q_-)$ is coprime to $p$, then there is a vertex $q$ of $Q$ such that $\d_Q(q)$ is coprime to $p$.
\end{lemm}

\red{We will prove this lemma later. It suffices to prove the following for our purpose in Case 3.} 

\begin{prop} \label{prop:5-4}
Suppose that $Q$ is face $p$-acyclic, has the same face poset as $\Delta^{n-1}\times [-1,1]$ and $\m(Q)$, $\m(Q_\pm)$ are coprime to $p$. Then $H^{odd}(X(Q,\v);\Z/p)=0$.
\end{prop}

\begin{proof}
We abbreviate $X(Q,\v)$ as $X$ and denote by $X_\epsilon$ ($\epsilon=+$ or $-$) the inverse image of $Q_\epsilon$ by the quotient map $\q\colon X\to Q$. Since $Q_\epsilon$ is face $p$-acyclic, has the same face poset as $\Delta^{n-1}$ and $\m(Q_\epsilon)$ is coprime to $p$ by assumption, we have
\begin{equation} \label{eq:5-8}
H^{odd}(X_\epsilon;\Z/p)=0
\end{equation}
by Proposition~\ref{prop:5-2}. 

By Lemma~\ref{lemm:5-1} there is a vertex $q$ of $Q$ such that $\d_Q(q)$ is coprime to $p$. Without loss of generality we may assume $q=\qn$, i.e. $\d_Q(\qn)$ is coprime to $p$. Since we have \eqref{eq:5-8} and the exact sequence 
\begin{equation*} 
\to H^{odd}(X/X_+;\Z/p)\to H^{odd}(X;\Z/p)\to H^{odd}(X_+;\Z/p)\to,
\end{equation*}
it suffices to prove 
\begin{equation} \label{eq:5-9}
H^{odd}(X/X_+;\Z/p)=0. 
\end{equation}

We consider two cases. 

\smallskip
{\it Case a.} The case where $\det(\v_1,\dots,\v_n)\not=0$. In this case, the characteristic function $\v$ on $Q$ induces a characteristic function on $Q/Q_+$, denoted $\v^+$, and $X/X_+=X(Q/Q_+,\v^+)$. We note that $Q/Q_+$ is face $p$-acyclic and has the same face poset as $\Delta^n$ \red{since $Q$ is face $p$-acyclic and has the same poset as $\Delta^{n-1}\times[-1,1]$}. Moreover, since $\qn$ is a vertex of $Q/Q_+$ and $\d_{Q/Q_+}(\qn)=\d_Q(\qn)$ is coprime to $p$, $\m(Q/Q_+)$ is coprime to $p$. Therefore, \eqref{eq:5-9} follows from Proposition~\ref{prop:5-2}. 

\smallskip
{\it Case b.} The case where $\det(\v_1,\dots,\v_n)=0$. 

\medskip
\noindent
{\bf Claim.} There is a vertex $q$ of $\Qn$ such that $\d_{\Qn}(q)$ is coprime to $p$, so $\m(\Qn)$ is coprime to $p$. 

\smallskip
\noindent
{\it Proof.} Write $\v_i=(\v_i^1,\dots,\v_i^n)^T$ and $\v_-=(\v_-^1,\dots,\v_-^n)^T$. Since $\v_n$ is primitive, we may assume $\v_n=(0,\dots,0,1)^T$ by Theorem~\ref{theo:4-1}. Denote by $\bv_i$ and $\bv_-$ the projection images of $\v_i$ and $\v_-$ on $\Z^n/\langle \v_n\rangle$ and by $\v'_i$ and $\v'_-$ the primitive vectors which have the same directions as $\bv_i$ and $\bv_-$ respectively. Then 
\[
\d_{\Qn}(q_i^-)=|\det(\v'_1,\dots,\widehat{\v'_i},\dots,\v'_{n-1},\v'_-)| 
\]
by definition and hence 
\begin{equation} \label{eq:5-9-1}
\d_{\Qn}(q_i^-)\big|\det(\bv_1,\dots,\widehat{\bv_i},\dots,\bv_{n-1},\bv_-).
\end{equation} 
On the other hand, since $\v_n=(0,\dots,0,1)^T$, we have
\[
\det(\v_1,\dots,\v_n)=\det(\bv_1,\dots,\bv_{n-1})
\]
and the left hand side above is zero by assumption. It follows that 
\red{\[
\begin{split}
\d_Q(\qn)=&|\det(\v_1,\dots,\v_{n-1},\v_-)|\\
=&|\v_-^n\det(\bv_1,\dots,\bv_{n-1})+\sum_{j=1}^{n-1}\v_{j}^{n}(-1)^{n-j}\det(\bv_1,\dots,\widehat{\bv_{j}},\dots,\bv_{n-1},\bv_-)|\\
=&|\sum_{j=1}^{n-1}\v_{j}^{n}(-1)^{n-j}\det(\bv_1,\dots,\widehat{\bv_{j}},\dots,\bv_{n-1},\bv_-)|
\end{split}
\]}
where the second identity above is the expansion of $\det(\v_1,\dots,\v_{n-1},\v_-)$ with respect to the $n$th row. 
By \eqref{eq:5-9-1} $\gcd(\d_{\Qn}(q_1^-),\dots,\d_{\Qn}(q_{n-1}^-))$ divides the last term above. Since $\d_Q(\qn)$ is coprime to $p$, this means that $\d_{\Qn}(q_i^-)$ is coprime to $p$ for some $i$, proving the claim. 

\medskip
Now we shall prove \eqref{eq:5-9} by induction on the dimension $n$ of $Q$. When $n=1$, $Q$ is a closed interval, $X$ is $S^2$ and $X_+$ is a point; so \eqref{eq:5-9} holds in this case. We assume $n\ge 2$ in the following. Let $\Xn$ be the inverse image of $\Qn$ by the quotient map $\q\colon X\to Q$. The face poset of $\Qn$ is the same as that of $\Delta^{n-2}\times [-1,1]$ and $\Qn$ is face $p$-acyclic. The facets corresponding to $\Delta^{n-2}\times \{\pm 1\}$ are $\Qn\cap Q_\pm$ and $\m(\Qn\cap Q_\pm)$ are coprime to $p$ by \eqref{eq:5-1} because $\m(Q_\pm)$ are coprime to $p$ by assumption. \red{Moreover,  $\m(\Qn)$ is also coprime to $p$ by the claim above}. Therefore
\begin{equation} \label{eq:5-10}
H^{odd}(\Xn/(\Xn\cap X_+);\Z/p)=0
\end{equation}
by the induction assumption. 

The quotient $Q/(\Qn\cup Q_+)=:\bQ$ is face $p$-acyclic and $\bQ$ has the same face poset as \red{$\lozenge^n$}. The characteristic function $\v$ on $Q$ induces a characteristic function on $\bQ$, denoted $\tv$, because $\qn$ is a vertex of $\bQ$ and $\d_{\bQ}(\qn)=\d_Q(\qn)$ is coprime to $p$, in particular nonzero. The quotient space $\Xn/(\Xn\cap X_+)$ is a subspace of $X/X_+$ and 
\begin{equation} \label{eq:5-11}
(X/X_+)\Big/\big(\Xn/(\Xn\cap X_+)\big)=X(\bQ,\tv).
\end{equation}
Since $\d_{\bQ}(\qn)=\m(\bQ)$ is coprime to $p$, $H^{odd}(X(\bQ,\tv);\Z/p)=0$ by Proposition~\ref{prop:5-1}. This together with \eqref{eq:5-11}, \eqref{eq:5-10} and the exact sequence 
\[
\begin{split}
&\to H^{odd}((X/X_+)\Big/\big(\Xn/(\Xn\cap X_+)\big);\Z/p)\to H^{odd}(X/X_+;\Z/p)\\
&\to H^{odd}(\Xn/(\Xn\cap X_+);\Z/p)\to
\end{split}
\]
implies \eqref{eq:5-9}.
\end{proof}

\red{Now it remains to prove Lemma~\ref{lemm:5-1}.}

\begin{proof}[Proof of Lemma~\ref{lemm:5-1}]
We may assume that $\m(Q_+)$ is coprime to $p$. We may also assume that $v_+=(0,\dots,0,1)^T$ by Theorem~\ref{theo:4-1} through some identification of $N$ with $\Z^n$. Suppose that 
\begin{equation} \label{eq:5-2}
\text{$p\big|\d_Q(q)$ for all vertices $q$ of $Q$} 
\end{equation}
and we will deduce a contradiction in the following. 

By Lemma~\ref{lemm:4-2}, $\det(v_1,\dots,v_n)$ is divisible by $\gcd(\d_Q(q_1^\epsilon),\dots,\d_Q(q_n^\epsilon))$, so it follows from \eqref{eq:5-2} that 
\begin{equation} \label{eq:5-3}
p\big|\det(\v_1,\dots,\v_n).
\end{equation} 
We write $v_i=(v_i^1,\dots,v_i^n)^T\in \Z^n$ for $i=1,2,\dots,n$. 

\medskip
\noindent
{\bf Claim 1.} There is an $i\in [n]$ such that $p\big|\v_i^j$ for all $j\not=n$. 

\smallskip
\noindent
{\it Proof.} 
Since $v_+=(0,\dots,0,1)^T$, we naturally identify the quotient lattice $\Z^n/\langle v_+\rangle$ with $\Z^{n-1}$ and then the projection image $\bv_i$ of $v_i$ on the quotient lattice $\Z^{n-1}$ is $(v_i^1,\dots,v_i^{n-1})$. Set $s_i=\gcd(v_i^1,\dots,v_i^{n-1})$. Then $\bv_i/s_i=:v_i'$ is primitive. Since $\d_Q(q)$ is assumed to be divisible by $p$ for all vertices $q$ of $Q$, we have 
\begin{equation} \label{eq:5-4}
p\big| \det(\v_{i_1},\dots,\v_{i_{n-1}},\v_+) \quad\text{for every subset $\{i_1,\dots,i_{n-1}\}$ of $[n]$}.
\end{equation} 
Here, since $\v_+=(0,\dots,0,1)^T$, we have 
\begin{equation} \label{eq:5-5}
\det(\v_{i_1},\dots,\v_{i_{n-1}},\v_+)=\det(\bv_{i_1},\dots,\bv_{i_{n-1}})=(\prod_{k=1}^{n-1}s_{i_k})\det(\v_{i_1}',\dots,\v_{i_{n-1}}').
\end{equation}

Now suppose that $s_i$ is not divisible by $p$ for any $i$. Then it follows from \eqref{eq:5-4} and \eqref{eq:5-5} that $p\big|\det(\v_{i_1}',\dots,\v_{i_{n-1}}')$ for every subset $\{i_1,\dots,i_{n-1}\}$ of $[n]$. Since $\m(Q_+)$ agrees with the greatest common divisor of all $\det(\v_{i_1}',\dots,\v_{i_{n-1}}')$ by \eqref{eq:4-1}, this shows that $p\big|\m(Q_+)$ which contradicts the assumption that $\m(Q_+)$ is coprime to $p$. Therefore $p\big|s_i$ for some $i$, proving the claim. 

\medskip
\noindent
{\bf Claim 2.} $p\big|\det(\v_{i_1},\dots,\v_{i_{n-2}},\v_-,\v_+)$ for every subset $\{i_1,\dots,i_{n-2}\}$ of $[n]$.

\smallskip
\noindent
{\it Proof.} Since $\v_+=(0,\dots,0,1)^T$, we have 
\begin{equation} \label{eq:5-6}
\det(\v_{i_1},\dots,\v_{i_{n-2}},\v_-,\v_+)=\det(\bv_{i_1},\dots,\bv_{i_{n-2}},\bv_-)
\end{equation}
where $\bv_-=(\v_-^1,\dots,\v_-^{n-1})^T$ is the projection image of $\v_-$ on the quotient $\Z^n/\langle \v_+\rangle=\Z^{n-1}$. We shall observe that the right hand side in \eqref{eq:5-6} is divisible by $p$. Without loss of generality we may assume that the $i$ in Claim 1 is $n$, so that $p\big|\v_n^j$ for all $j\not=n$. We consider two cases. 

{\it Case a.} The case where $n\in \{i_1,\dots,i_{n-2}\}$. Since $\bv_n=(\v_n^1,\dots,\v_n^{n-1})^T$ and $p\big|\v_n^j$ for all $j\not=n$, the right hand side in \eqref{eq:5-6} is divisible by $p$.

{\it Case b.} The case where $n\notin \{i_1,\dots,i_{n-2}\}$. In this case, we consider the expansion of $\det(\v_{i_1},\dots,\v_{i_{n-2}},\v_-,\v_n)$ with respect to the last column. Since $\v_n=(\v_n^1,\dots,\v_n^{n})^T$ and $p\big|\v_n^j$ for all $j\not=n$, we have
\red{ 
\begin{equation} \label{eq:5-7}
|\det(\v_{i_1},\dots,\v_{i_{n-2}},\v_-,\v_n)|\equiv |\v_n^n\det(\bv_{i_1},\dots,\bv_{i_{n-2}},\bv_-)| \pmod{p}.
\end{equation}}
Here the left hand side above is $\d_Q(q)$ for $q=(\bigcap_{k=1}^{n-2}Q_{i_k})\cap Q_-\cap Q_n$, so it is divisible by $p$ by \eqref{eq:5-2}. Moreover, $\v_n^n$ is not divisible by $p$ because otherwise every entry of $\v_n$ is divisible by $p$ and this contradicts $\v_n$ being primitive. It follows from \eqref{eq:5-7} that the right hand side in \eqref{eq:5-6} is divisible by $p$ in this case, too. 

This completes the proof of the claim. 

\medskip

Now \eqref{eq:5-2}, \eqref{eq:5-3} and Claim 2 show that all $n\times n$ minors of $(\v_1,\dots,\v_n,\v_-,\v_+)$ are divisible by $p$ and hence $p\big|\m(Q)(=|N/\hN|)$ by Theorem~\ref{theo:4-1}. This contradicts the assumption that $\m(Q)$ is coprime to $p$, proving the lemma. 
\end{proof} 

\section{Example} \label{sect:6}

In this section we shall give an example of a compact simplicial toric variety showing that the converse of Proposition \ref{prop:3-1} does not hold in general.

Let $Q$ be the 3-dimensional simple polytope with the $7$ facets $Q_+,~ Q_-,$ $Q_1,..., Q_5$, where $Q_4$ and $Q_5$ are triangles obtained by cutting two vertices of a prism, shown in Figure \ref{fig:6-1} below. The polytope $Q$ can be obtained from \red{$\lozenge^3$} by performing a vertex cut four times. 

\begin{figure}[htbp]
\begin{center}
\begin{tikzpicture}
\draw[thick](-2,2)--(-0.6666,1.3333)--(0.6666,1.3333)--(2,2)--cycle;
\draw[thick](-2,-1)--(-0.6666,-1.6)--(0,-0.7)--(0,.4)--(-0.6666,1.3333)--(-2,2)--cycle;
\draw[thick](0,.4)--(0,-0.7)--(0.5,-1.6)--(2,-1)--(2,2)--(0.6666,1.3333)--cycle;
\draw[thick,dashed](-2,-1)--(2,-1);
\draw[thick](0.5,-1.6)--(-0.6666,-1.6)--(0,-0.7)--cycle;
\fill (-0.6666,1.3333) circle(3pt) ;
\fill (0.6666,1.3333) circle(3pt) ;
\fill (0,.4)circle(3pt) ; 
\fill (0,-0.7)circle(3pt) ;
\fill (0.5,-1.6)circle(3pt) ;
\fill (-0.6666,-1.6)circle(3pt) ;
\fill (2,-1) circle(3pt);
\fill (-2,-1) circle(3pt);
\fill (-2,2) circle(3pt);
\fill (2,2) circle(3pt);
\coordinate [label=left:{$Q_+$}] ($Q_+$) at (0.3,1.65);
\coordinate [label=left:{$Q_4$}] ($Q_4$) at (0.3,0.95);
\coordinate [label=left:{$Q_5$}] ($Q_5$) at (0.3,-1.25);
\coordinate [label=left:{$Q_2$}] ($Q_2$) at (-0.75,0.2);
\coordinate [label=left:{$Q_1$}] ($Q_1$) at (1.55,0);
\draw[thick](3.5,1) ..controls (3,1.3) and (2.8,1.35) ..(2, 1.4);
\draw[dashed,->](2, 1.4)..controls(1.5,1.3)..(1.2,1);
\coordinate [label=right:{$Q_3$}] ($Q_3$) at (3.5,0.9);
\draw[thick](2.6,-2.4)..controls(2.0,-2.61)and(1.43,-3.)..(0.85,-1.5);
\draw[dashed,->](0.85,-1.5)..controls(0.8,-1.35)..(0.8,-1.15);
\coordinate [label=right:{$Q_-$}] ($Q_-$) at (2.6,-2.2);
\end{tikzpicture}
\caption{}\label{fig:6-1}
\end{center}
\end{figure} 

Let $d$ be a positive integer.
To the $7$ facets $Q_1,..., Q_5$, $Q_+,~Q_-$, we respectively assign the following vectors 
\begin{align*}
&v_1=(1,0,0)& &v_2=(-1,d,-d) & &v_3=(-1,-d,0)\\
&v_4=(0,1,0)& &v_5=(d,1-d,-d) & &\\
&v_+=(0,0,1)& &v_-=(1,-1,-1),& &
\end{align*}
giving a characteristic function $\v$ on $Q$. There are ten vertices in $Q$. At each vertex, there are exactly three facets meeting and the determinant of the three vectors assigned to the facets is nonzero, indeed their absolute values are as follows:
\[
\begin{split}
&|\det(\v_1, \v_4,\v_+)|=|\det(\v_2,\v_4,\v_+)|=|\det(\v_1,\v_5,\v_-)|=1\\
&|\det(\v_1,\v_2,\v_4)|=|\det(\v_1,\v_3,\v_+)|=|\det(\v_1,\v_3,\v_-)|=d\\
& |\det(\v_1,\v_2,\v_5)|=d(2d-1)\ \ \quad\quad |\det(\v_2,\v_5,\v_-)|=d+1\\
&|\det(\v_2,\v_3,\v_-)|=d(d+3) \quad\qquad|\det(\v_2,\v_3,\v_+)|=2d.
\end{split}
\]
(Precisely speaking, the vectors are regarded as column vectors here by taking transpose.) Therefore, at each vertex, the cone spanned by the three vectors is 3-dimensional and has the origin as the apex. One can also check that 
\begin{align*}
&\v_4=(\v_1+\v_2+d\v_+)/d& &\v_5=\big((d+1)\v_1+\v_2+d(2d-1)\v_-\big)/2d\\
&v_+=-(2\v_1+\v_2+\v_3)/d & &v_-=((d+3)\v_1+\v_2+2\v_3)/d. 
\end{align*}
Since $d$ is a positive integer, this shows that $-\v_+$ is in the cone
$\angle \v_1\v_2\v_3$ and $\v_4$ is in the cone $\angle \v_1\v_2\v_+$
while $\v_-$ is in the cone $\angle \v_1\v_2\v_3$ and $\v_5$ is in the
cone $\angle \v_1\v_2\v_-$ (see Figure 2), where $\angle uvw$ denotes the cone spanned by vectors $u,v,w$. This implies that the ten 3-dimensional cones have no overlap and cover the entire $\R^3$, giving a complete simplicial fan so that the quotient space $X=X(Q,v)$ is homeomorphic to a compact simplicial toric variety. 

\red{
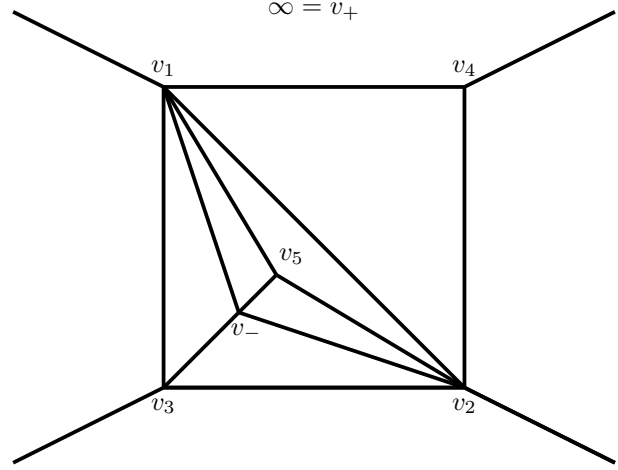
\begin{figure}[htbp]
\begin{center}
\begin{tikzpicture}[scale=1,line width=0.5mm]
\draw (0,0)--(2,1)--(6,1)--(8,0);
\draw (2,1)--(2,5)--(0,6);
\draw (2,5)--(6,5)--(8,6);
\draw (6,5)--(6,1)--(8,0);
\draw (6,1)--(2,1);
\draw (2,1)--(3.5,2.5);
\draw (6,1)--(3,2);
\draw (3,2)--(2,5);
\draw (2,5)--(6,1);
\draw (3.5,2.5)--(2,5);
\draw (3.5,2.5)--(6,1); 
\draw (2,1) node[below]{$v_3$};
\draw (6,1) node[below]{$v_2$};
\draw (3.1,2) node[below]{$v_-$};
\draw (3.7,2.5) node[above]{$v_5$};
\draw (2,5) node[above]{$v_1$};
\draw (6,5) node[above]{$v_4$};
\draw (4,6) node {$\infty=v_+$};
\end{tikzpicture}
\caption{ \red{Each vector $v_i$ is denoted by a point in $\R^2\cup\{\infty\}$ and a segment connecting $v_i,v_j$ corresponds to the 2-dimensional cone spanned by them and a triangle formed by $v_i,v_j,v_k$ corresponds to the 3-dimensional cone spanned by them.} }
\label{fig:6-2}
\end{center}
\end{figure}}

We shall check that $\mu(Q_I)=1$ for each face $Q_I$ of $Q$, where $\mu(Q_I)$ is defined in Section~\ref{sect:3}. 
As remarked in Section~\ref{sect:3}, $\m(Q_I)=1$ when $|I|=2$ or $3$. Clearly $\hN=N(=\Z^3)$. Therefore it suffices to check $\m(Q_I)=1$ when $|I|=1$. At vertices $Q_1\cap Q_4\cap Q_+$, $Q_2\cap Q_4\cap Q_+$ and $Q_1\cap Q_5\cap Q_-$, we have 
\[
|\det(\v_1, \v_4,\v_+)|=|\det(\v_2,\v_4,\v_+)|=|\det(\v_1,\v_5,\v_-)|=1
\]
and hence $\m(Q_I)=1$ for every $I$ with $|I|=1$ except $I=\{3\}$ again by the remark in Section~\ref{sect:3}. In order to see $\m(Q_3)=1$, we note that $\{\v_3, \v_4, \v_+\}$ is a base of $N$ and 
\[
\v_1=-v_3-d\v_4,\quad \v_2=\v_3+2d\v_4-d\v_+.
\]
Therefore, the images of $\v_1$ and $\v_2$ by the quotient map $\pi_{\{3\}}\colon N\to N(\{3\})=N/\langle \v_3\rangle$ are $(-d,0)$ and $(2d,-d)$ with respect to the base $\{\pi_{\{3\}}(\v_4), \pi_{\{3\}}(\v_+)\}$. Thus the corresponding primitive vectors are $(-1,0)$ and $(2,-1)$ which form a base of $N(\{3\})$. Hence $\mu(Q_3)=1$. 

We shall compute $H^3(X)$. Take a plane in $\R^3$ which meets the
facets $Q_1,Q_2,Q_3$ transversally and does not meet the other facets of
$Q$. Cutting $Q$ along the plane, we divide $Q$ into two polytopes, denoted $P_+$ and $P_-$ containing $Q_+$ and $Q_-$ respectively. Let $\pi\colon X\to Q$ be the quotient map and set 
\[
Y_\epsilon:=\pi^{-1}(P_\epsilon) \ \ \text{for $\epsilon=\pm$},\quad Y:=Y_+\cap Y_-, \quad P:=P_+\cap P_-.
\] 
The quotient space $P_\epsilon/P$ can be regarded as a prism. The characteristic function $\v$ on $Q$ induces a characteristic function on $P_\epsilon/P$, denoted $w_\epsilon$, and $X/Y_+=Y_-/Y$ (resp. $X/Y_-=Y_+/Y$) is homeomorphic to $X(P_-/P,w_-)$ (resp. $X(P_+/P,w_+)$). The same argument as above shows that $\mu$ takes $1$ on all faces of the prism $P_\epsilon/P$, so 
\begin{equation} \label{eq:6-1}
\text{$H^{*}(X, Y_\epsilon)$ and $H^{*}(Y_\epsilon,Y)$ are torsion free and vanish in odd degrees}
\end{equation}
by Proposition~\ref{prop:5-4}. 

Let $\tilde Q$ be a nice manifold with corners obtained from $Q$ by
collapsing $Q_4\cup Q_+$ and $Q_5\cup Q_-$ to a point respectively. The $\tilde Q$ has three facets coming from $Q_1,Q_2,Q_3$ and the characteristic function $\v$ on $Q$ induces a characteristic function $\tv$ on $\tilde Q$. Since 
\[
v_1=(1,0,0),\quad v_2=(-1,d,-d),\quad v_3=(-1,-d,0),
\]
one can see that $H^4(X(\tilde Q,\tv))\cong \Z/d$ by Corollary~\ref{coro:2-3}, and since $X(\tilde Q,\tv)$ is homeomorphic to the suspension of $Y$, we obtain 
\begin{equation} \label{eq:6-2}
H^3(Y)\cong \Z/d.
\end{equation}

Now, consider the exact sequence in cohomology for the pair $(Y_+, Y)$:
\begin{equation}\label{eq:6-3}
\to H^3(Y_+,Y)\to H^3(Y_+)\to H^3(Y)\to H^4(Y_+,Y)\to.
\end{equation}
Since $H^3(Y_+,Y)=0$ and $H^4(Y_+,Y)$ is torsion free by \eqref{eq:6-1}
and $H^3(Y)$ is a torsion group by \eqref{eq:6-2}, it follows from the exact sequence \eqref{eq:6-3} that 
\begin{equation} \label{eq:6-4}
H^3(Y_+)\cong H^3(Y)\cong \Z/d.
\end{equation}
Next, consider the exact sequence in cohomology for the pair $(X,Y_+)$: 
\begin{equation}\label{eq:6-5}
\to H^3(X,Y_+)\to H^3(X)\to H^3(Y_+)\to H^4(X,Y_+)\to.
\end{equation}
Similarly to the above argument, $H^3(X,Y_+)=0$ and $H^4(X,Y_+)$ is
torsion free by \eqref{eq:6-1} and $H^3(Y_+)$ is a torsion group by \eqref{eq:6-4}, so it follows from the exact sequence \eqref{eq:6-5} that 
\[
H^3(X)\cong H^3(Y_+)\cong \Z/d.
\]
Thus $X=X(Q,\v)$ is the desired example when $d\ge 2$.

\section*{Appendix}

In this appendix, we observe that when $X$ is a compact simplicial toric variety of complex dimension $n$, a result of Fischli \cite{fisc92} or Jordan \cite{jord98} implies that $H^{2n-1}(X)\cong N/\hN$ and $\Tor H^{2n-2}(X)\cong \wedge^2N/(\hN\wedge N)$, where $\Tor H^{2n-2}(X)$ denotes the torsion part of $H^{2n-2}(X)$. This result agrees with Proposition~\ref{prop:2-1} since $Q$ is contractible in this case. 

Let $\Delta$ be a simplicial complete fan of dimension $n$ and let $X$ be the associated compact simplicial toric variety. Let $M$ be the free abelian group dual to $N$. \red{Since $N=\Hom(S^1,T)$, $M$ can be thought of as $\Hom(T,S^1)$.} According to \cite[Theorem 2.3]{fisc92} or \cite[Theorem 2.5.5]{jord98}, 
\[
H^{2n-1}(X)\cong \coker \delta_1,\quad \Tor H^{2n-2}(X)\cong \coker\delta_2,
\]
where 
\begin{equation} \label{eq:a-0}
\delta_r\colon \bigoplus_{\tau\in\Delta^{(1)}}\wedge^{n-r}(\tau^{\perp}\cap M)\to \wedge^{n-r}M \qquad(r=1,2)
\end{equation}
is the sum of inclusion maps \red{with signs}, $\Delta^{(1)}$ denotes the set of one-dimensional cones in $\Delta$ and $\tau^\perp$ denotes the subspace of $M\otimes\R$ which vanish on $\tau$. 

We shall interpret the above in terms of $N$. Let $\sigma$ be a cone of dimension $n-k$ in $\Delta$. Then we have 
\begin{equation} \label{eq:a-1}
\begin{split}
\wedge^\ell(\sigma^\perp\cap M)&\cong \Hom(\wedge^{k-\ell}(\sigma^\perp\cap M),\Z)\quad \text{($\because \rank \sigma^\perp\cap M=k$)}\\
&\cong \wedge^{k-\ell}(N/N_\sigma)\qquad \text{($\because$ $N/N_\sigma$ is dual to $\sigma^\perp\cap M$)}\\
&\cong (\wedge^{n-k}N_\sigma)\wedge(\wedge^{k-\ell}N)
\end{split}
\end{equation}
where $N_\sigma$ is the intersection of $N$ with the subspace of $N\otimes\R$ spanned by $\sigma$. 
The last isomorphism above is given as follows. Choose a base $\rho_1,\dots,\rho_{n-k}$ of $N_\sigma$. Since $N_\sigma$ is of rank $n-k$, $\wedge^{n-k}N_\sigma$ is a free abelian group of rank one and $\rho_1\wedge\dots\wedge\rho_{n-k}$ is its generator. For $w\in N$, we denote by $[w]$ the element of $N/N_\sigma$ determined by $w$. Then the following correspondence 
\[
[w_1]\wedge\dots\wedge[w_{k-\ell}]\to \rho_1\wedge\dots\wedge\rho_{n-k}\wedge w_1\wedge\dots\wedge w_{k-\ell}
\]
is well defined and gives the desired isomorphism from $\wedge^{k-\ell}(N/N_\sigma)$ to $(\wedge^{n-k}N_\sigma)\wedge(\wedge^{k-\ell}N)$. This isomorphism is independent of the choice of the base $\rho_1,\dots,\rho_{n-k}$ up to sign. Namely, the isomorphism \eqref{eq:a-1} depends only on the choice of orientations on $M$ (or $N$) and $\sigma$. 

Applying \eqref{eq:a-1} to $\sigma=\tau\in \Delta^{(1)}$ and $\sigma=0$, we obtain 
\begin{alignat*}{2}
&\wedge^{n-1}(\tau^{\perp}\cap M)\cong N_\tau,&\qquad &\wedge^{n-1}M\cong N,\\
&\wedge^{n-2}(\tau^{\perp}\cap M)\cong N_\tau\wedge N,&\qquad &\wedge^{n-2}M\cong \wedge^2 N. 
\end{alignat*}
Since $\delta_r$ is the sum of inclusion maps \red{with signs}, 
the image of $\delta_1$ (resp. $\delta_2$) in \eqref{eq:a-0} can be identified with $\hat N$ (resp. $\hN\wedge N$) and hence
\[
H^{2n-1}(X)\cong E_2^{n,n-1}\cong N/\hN,\quad
\Tor H^{2n-2}(X)\cong E_2^{n,n-2}\cong \wedge^2N/(\hat N\wedge N).
\]

\bigskip
\noindent
{\bf Acknowledgment.} 
We thank Tony Bahri, Soumen Sarkar and Jongbaek Song for their interest and useful comments on the paper. We also thank Matthias Franz for his comments and for his development of the Maple package torhom which was very useful in our research. 
\red{Finally we thank the anonymous referee for helpful comments to improve the presentation of the paper.}

\end{document}